\documentclass[12pt,oneside,final,reqno]{amsart}

\usepackage[utf8]{inputenc}
\usepackage[T1]{fontenc}

\usepackage{color}
\usepackage{amsmath, amssymb, amsthm}
\usepackage[australian]{babel}
\usepackage{etoolbox}
\usepackage{fancyhdr}
\usepackage{graphicx}
\usepackage{hyperref}
\usepackage{mathtools}
\usepackage{microtype}
\usepackage{stmaryrd}
\usepackage{verbatim}

\numberwithin{equation}{section}

\theoremstyle{plain}
	\newtheorem*{thm*}			{Theorem}
	\newtheorem*{conj*}			{Conjecture}
	\newtheorem*{cor*}			{Corollary}
	\newtheorem*{lem*}			{Lemma}
	\newtheorem*{prob*}			{Problem}
	\newtheorem*{prop*}			{Proposition}
\theoremstyle{definition}
	\newtheorem*{df*}			{Definition}
	\newtheorem*{ex*}			{Example}
	\newtheorem*{xc*}			{Exercise}

\theoremstyle{plain}
	\newtheorem{thm}				{Theorem}[section]
	\newtheorem{conj}				[thm]{Conjecture}
	\newtheorem{cor}				[thm]{Corollary}

	\newtheorem{prop}			[thm]{Proposition}
    \newtheorem{rem}            [thm]{Remark}
\theoremstyle{definition}

\theoremstyle{definition}
		\newtheorem*{rem*}		{Remark}

\newcommand{\define}[5]{\expandafter#1\csname#5#2\endcsname{#3{#4{#5}}}}

\forcsvlist{\define{\DeclareMathOperator}{}{}{}}
	{ad,adj,codim,coker,disc,fix,id,im,lcm,nil,ob,ord,rad,rank,res,sgn,supp,tr,vol}
\forcsvlist{\define{\DeclareMathOperator}{}{}{}}
	{Ad,Alt,Ann,Ass,Aut,Coinf,Cor,Der,Div,Emb,End,Ext,Fr,Gal,Hom,Ind,Inf,Inn,Irr,Li,Lie,Mat,Max,Out,Pic,Proj,Res,Spec,Stab,Supp,Sym,Tor}

\forcsvlist{\define{\newcommand}{}{\mathrm}{}}{ab,al,diag,nr,op,red,sep,tor}
\forcsvlist{\define{\newcommand}{}{\mathbb}{}}{C,F,N,Q,R,Z}
\forcsvlist{\define{\newcommand}{bb}{\mathbb}{\uppercase}}{a,b,d,e,g,h,i,j,k,l,m,o,p,s,t,u,v,w,x,y}
\forcsvlist{\define{\newcommand}{cal}{\mathcal}{\uppercase}}{a,b,c,d,e,f,g,h,i,j,k,l,m,n,o,p,q,r,s,t,u,v,w,x,y,z}
\forcsvlist{\define{\newcommand}{frak}{\mathfrak}{\uppercase}}{a,b,c,d,e,f,g,h,i,j,k,l,m,n,o,p,q,r,s,t,u,v,w,x,y,z}
\forcsvlist{\define{\newcommand}{fr}{\mathfrak}{}}{a,b,c,d,e,f,g,h,i,j,k,l,m,n,o,p,q,r,s,t,u,v,w,x,y,z}
\forcsvlist{\define{\newcommand}{scr}{\mathscr}{\uppercase}}{a,b,c,d,e,f,g,h,i,j,k,l,m,n,o,p,q,r,s,t,u,v,w,x,y,z}
\forcsvlist{\define{\newcommand}{sf}{\mathsf}{\uppercase}}{a,b,c,d,e,f,g,h,i,j,k,l,m,n,o,p,q,r,s,t,u,v,w,x,y,z}

\DeclareMathOperator\UL{U}
	
\forcsvlist{\define{\DeclareMathOperator}{}{}{}}{GL,PGL,PSL,SL,SO,SU,Sp,USp}
\forcsvlist{\define{\newcommand}{fr}{\mathfrak}{}}{gl,sl,so,sp,su,usp}

\let\Re\relax\DeclareMathOperator\Re 		{Re}

\newcommand{\dfeq}				{\coloneqq}

\newcommand{\nn}					{\nonumber\\}

	\newcommand{\subeq}			{\subseteq}
	
\newcommand{\ud}					{\mathrm{d}}

\newcommand{\pa}[1]				{\left(			#1			\right)}

\renewcommand{\bar}[1]				{\overline{#1}}

\newcommand{\osc}{\mathrm{osc}}
\usepackage{fullpage}
\usepackage{latexsym}
\usepackage{mathrsfs}
\usepackage{times}

\renewcommand{\hat}[1]			{\widehat{#1}}

\newcommand{\be}{\begin{equation}}
	\newcommand{\ee}{\end{equation}}
\newcommand{\bea}{\begin{eqnarray}}
	\newcommand{\eea}{\end{eqnarray}}
\newcommand{\ben}{\begin{enumerate}}
	\newcommand{\een}{\end{enumerate}}
\newcommand{\bi}{\begin{itemize}}
	\newcommand{\ei}{\end{itemize}}

\newcommand{\even}{\mathrm{even}}

\newcommand{\burl}[1]{\textcolor{blue}{\url{#1}}}
\newcommand{\integral}[1]{\int_{-#1}^{+#1}}

\renewcommand{\geq}{\geqslant}
	\renewcommand{\leq}{\leqslant}
\renewcommand{\mod}[1]
{
	{ 	\ifmmode\text{\rm\ (mod~$#1$)}\else
		\discretionary{}{}{\hbox{ }}\rm(mod~$#1$)\fi
	}
}

\DeclareMathOperator\Var{Var}

\begin{document}

%	Title block

\title[Zeros of Dirichlet $L$-Functions over Function Fields]
{ 	Zeros of Dirichlet $L$-Functions over Function Fields
}
%	\dedicatory{}
\date{\today}	%%%
\thanks{This research took place at the 2013 SMALL REU at Williams College. The authors were supported by the National Science Foundation, under grant DMS0850577. The first named author was also supported by a postdoctoral fellowship from IH\'ES and an EPSRC William Hodge Fellowship, and the second named author by NSF grant DMS1265673. The authors would like to thank Jon Keating and Ze\'{e}v Rudnick for helpful comments on an earlier draft.}

\begin{abstract}
Random matrix theory has successfully modeled many systems in physics and mathematics, and often the analysis and results in one area guide development in the other. Hughes and Rudnick computed $1$-level density statistics for low-lying zeros of the family of primitive Dirichlet $L$-functions of fixed prime conductor $Q$, as $Q  \to \infty$, and verified the unitary symmetry predicted by random matrix theory. We compute $1$- and $2$-level statistics of the analogous family of Dirichlet $L$-functions over $\mathbb{F}_q(T)$. Whereas the Hughes-Rudnick results were restricted by the support of the Fourier transform of their test function, our test function is periodic and our results are only restricted by a decay condition on its Fourier coefficients. We show the main terms agree with unitary symmetry, and also isolate error terms. In concluding, we discuss  an $\F_q(T)$-analogue of Montgomery's Hypothesis on the distribution of primes in arithmetic progressions, which Fiorilli and Miller show would remove the restriction on the Hughes-Rudnick results.
\end{abstract}

%	Author block(s)

\author{Julio C. Andrade}
\address{Institut des Hautes \'{E}tudes Scientifiques (IH\'{E}S), Bures-sur-Yvette, France}
\email{\href{mailto:j.c.andrade@ihes.fr}{\textcolor{blue}{j.c.andrade@ihes.fr}}}

\author{Steven J. Miller}
\address{Department of Mathematics and Statistics, Williams College,
Williamstown, MA 01267}
\email{\href{mailto:sjm1@williams.edu}{\textcolor{blue}{sjm1@williams.edu}}, \href{mailto:Steven.Miller.MC.96@aya.yale.edu}{\textcolor{blue}{Steven.Miller.MC.96@aya.yale.edu}}}

\author{Kyle Pratt}
\address{Department of Mathematics, Brigham Young University, Provo, UT 84602}
\email{\href{mailto:kyle.pratt@byu.net}{\textcolor{blue}{kyle.pratt@byu.net}}, \href{mailto:kvpratt@gmail.com}{\textcolor{blue}{kvpratt@gmail.com}}}

\author{Minh-Tam Trinh}
\address{Department of Mathematics, Princeton University, Princeton, NJ 08544}
\email{\href{mailto:mtrinh@princeton.edu}{\textcolor{blue}{mtrinh@princeton.edu}}, \href{mailto:mqtrinh@gmail.com}{\textcolor{blue}{mqtrinh@gmail.com}}}

\maketitle	%%%	Must go after the title and author blocks!
\setcounter{tocdepth}{1}
\tableofcontents

%-------------------------------------------------------------------------

%%%%%%%%%%%%%%%%%%%%%%%%%%%%%%%%%%%%%%%%%%%%%%%%%%%%%%%%%%%%%%%%%%%%%%%%%%%%%%%%%%%%%%%%%%%%%%%%%%%%%%%%%%%%%%%%%%%%%%%%%%%%%
%%%%%%%%%%%%%%%%%%%%%%%%%%%%%%%%%%%%%%%%%%%%%%%%%%%%%%%%%%%%%%%%%%%%%%%%%%%%%%%%%%%%%%%%%%%%%%%%%%%%%%%%%%%%%%%%%%%%%%%%%%%%%
%%%%%%%%%%%%%%%%%%%%%%%%%%%%%%%%%%%%%%%%%%%%%%%%%%%%%%%%%%%%%%%%%%%%%%%%%%%%%%%%%%%%%%%%%%%%%%%%%%%%%%%%%%%%%%%%%%%%%%%%%%%%%
\section{Introduction}\label{sec:intro}

\subsection{Background}\label{subsec:background}

In the 1970s, Montgomery and Dyson conjectured that local statistics of critical zeros of the Riemann zeta function---in the limit of large height---should match those of angles of eigenvalues of matrices in the Gaussian Unitary Ensemble (GUE), which Wigner, Dyson and others (see \cite{FirM} for a historical overview) had already used with great success in modeling the energy levels of heavy nuclei. Their ideas, exemplified by the Pair Correlation Conjecture in \cite{Mo2}, began a long history of investigation into connections between number theory, physics and random matrix theory, which continue strong today (see for example \cite{CFZ1,CFZ2,CS}).

There have been many investigations, theoretical and numerical, on the zeros of the Riemann zeta function in different regimes. For example, Odlyzko\cite{Od1,Od2} checked various statistics of critical zeros of the Riemann zeta function high up on the critical line numerically, including pair correlation, and found extraordinary agreement with GUE predictions, and Berry \cite{Ber1,Ber2} made semi-classical predictions for these zeros in different ranges, again obtaining beautiful fits. While the leading order asymptotics of the zero statistics and the eigenvalues statistics are identical and asymptotically no factors of arithmetical nature appear, the work of Bogomolny and Keating \cite{BogKea1} has identified lower order terms in the pair correlation of the zeros of the Riemann zeta function. From their heuristics is clear that arithmetical contributions play a role in lower order terms. The same can be said about the two--point correlation function for Dirichlet L-functions \cite{BogKea2}. Katz and Sarnak extended this philosophy to families of $L$-functions in \cite{KS1,KS2}. They proposed that zeros of $L$-functions in suitable ``families'' would have similar statistics to each other, and that the statistics of a given family, in the limit of large analytic conductor, would match those of eigenangles of matrices in some classical compact group under Haar measure, in the limit of large dimension. Thus families of $L$-functions would correspond to one of three basic symmetry types: unitary, symplectic, or orthogonal. For recent discussions about a working definition of families of $L$--functions, as well as how to determine the underlying symmetry, see \cite{DM,ST,SST}.

Originally the Katz-Sarnak Conjectures were investigated in what we will call the \emph{local regime near the central point $s = 1/2$}; that is, in intervals around $s = 1/2$ shrinking as the conductor grows, so that the number of zeros it contains is roughly constant (see, among many others, \cite{ILS,HR,Mil1,Rub,Yo} for some of the earlier results in the field). In this regime the conjectures are very difficult, and most results are limited to test functions whose Fourier transforms have severely restricted support. Moreover, it is necessary to average over a ``family,'' as one $L$-function cannot have sufficiently many normalized zeros near the central point. This is in sharp contrast to other statistics such as the $n$-level correlation; these statistics study zeros high up on the critical line, and one $L$-function has sufficiently many zeros far from the central point to permit an averaging. For example, Rudnick and Sarnak \cite{RS} computed the $n$-level correlation for the zeroes of not just the Riemann zeta function but any cuspidal automorphic form for a restricted class of test functions, proving their expression agrees with the $n$-level correlation for the eigenvalues of random unitary matrices.

Our own work extends \cite{HR}, in which Hughes and Rudnick compute the $m$\textsuperscript{th} centered moment of the $1$-level density of the family of primitive Dirichlet $L$-functions of fixed conductor $Q$ as $Q \to \infty$, for test functions $\phi$ such that $\supp(\hat{\phi})  \subeq (-2/m, 2/m)$. This family should have unitary symmetry, which the authors verified for suitably restricted test functions. We consider analogous questions in the function field case. There has been significant progress in this area of late (see \cite{FR,Rud} among others); we briefly comment on one particular example which illuminates the contributions that function field results can have to random matrix theory and mathematical physics.

In his thesis Rubinstein \cite{Rub} showed the $n$-level density of quadratic Dirichlet $L$-functions agrees with the random matrix theory prediction of symplectic symmetry for support in $(-1/n, 1/n)$. In the course of his investigations he analyzed the combinatorial expansions for the $n$-level densities of the classical compact groups, though he only  needed the results for restricted test functions due to the limitations on the number theory calculations. Gao \cite{Gao} doubled the support on the number theory side in his thesis, but due to the complexity of the combinatorics was only able to show the two computed quantities agreed for $n \le 3$. (This is not the first time there has been difficulty comparing number theory and random matrix theory; see also Hughes and Miller \cite{HM}, where they derive an alternative to the determinant expansions from Katz-Sarnak \cite{KS1,KS2} which is more amenable for comparing $n$-level densities with support restricted as in the results in the literature.) Levinson and Miller \cite{LM} devised a new approach which allowed them to show agreement for $n\le 7$; unlike the ad-hoc method of Gao, they developed a canonical formulation of the quantities and reduced the general case to a combinatorial identity involving Fourier transforms. The work of Entin, Roditty-Gershon and Rudnick \cite{ER-GR} \emph{bypasses} these extremely difficulty calculations and obstructions; they are able to show agreement between number theory and random matrix theory by showing the number theory answer agrees with the answer of a related problem involving function fields, which are known to agree with random matrix theory by different techniques. In particular, their function field results prove the combinatorial identity and provide a new path through comparisons with random matrix theory.

Katz and Sarnak suggested that a possible motivation for their conjectures is the analogy between number fields and global function fields.
The Riemann zeta function and Dirichlet $L$-functions can be considered $L$-functions ``over $\Q$''; they possess analogues ``over $\F_q(T)$,'' which occur as factors of numerators of zeta functions of projective curves over $\F_q$. As proven by Deligne \cite{De}, the zeros of the latter have a spectral interpretation, as reciprocals of eigenvalues of the Frobenius endomorphism acting on $\ell$-adic cohomology. Katz-Sarnak \cite{KS1} proved agreement with GUE $n$-level correlation \emph{unconditionally} for the family of isomorphism classes of curves of genus $g$ over $\F_q$, in the limit as both $g, q \to \infty$. Their main tool was Deligne's result that the Frobenius conjugacy classes become equidistributed in the family's monodromy group as $q \to \infty$.

Recently, there has been interest in the ``opposite limit,'' where $q$ is held fixed and $g \to \infty$. In \cite{FR}, the authors considered zeros of zeta functions of hyperelliptic curves of genus $g$ over $\F_q$. Instead of looking at zeros in the \emph{local} regime, they look at zeros in (1) \emph{global} and (2) \emph{mesoscopic} regimes: that is, in intervals $\ical$ around $0$ such that either (1) $|\ical|$ is fixed, or (2) $|\ical| \to 0$ but $g|\ical| \to \infty$. In both regimes, they show that the zeros become equidistributed in $\ical$ as $g \to \infty$, and the  normalized fluctuations in the number of the zeros are Gaussian. Xiong \cite{Xi} extended their work to families of $\ell$-fold covers of $\pbb^1(\F_q)$, for prime $\ell$ such that $q \equiv 1\pmod{\ell}$, again obtaining Gaussian behavior. For other works in this direction see \cite{BDFL1, BDFL2, BDFLS}.

By making use of the ratios conjecture adapted to the function field setting Andrade and Keating \cite{AK} have derived the one-level density of the zeros of the family of quadratic Dirichlet $L$--functions over $\mathbb{F}_{q}(T)$ as $q$ is fixed and $g\rightarrow\infty$ with no restriction on the test function. Also in a recent paper Roditty--Gershon \cite{RG} have computed the averages of products of traces of high powers of the Frobenius class for the hyperelliptic ensemble and as consequence she was able to compute the $n$--level density (with restriction on the test function) of the zeros of quadratic Dirichlet $L$--functions over $\mathbb{F}_{q}(T)$. These results are similar as those obtained in this paper with the difference that in this paper we are using $L$--functions in a unitary family and the statistics that are being considered are somehow different (see next section for more details).

%%%%%%%%%%%%%%%%%%%%%%%%%%%%%%%%%%%%%%%%%%%%%%%%%%%%%%%%%%%%%%%%%%%%%%%%%%%%%%%%%%%555
%%%%%%%%%%%%%%%%%%%%%%%%%%%%%%%%%%%%%%%%%%%%%%%%%%%%%%%%%%%%%%%%%%%%%%%%%%%%%%%%%%%%
\subsection{Outline}\label{subsec:outline}

We study the $\F_q(T)$-analogue of the Dirichlet $L$-function family of \cite{HR} (we review their definitions and properties in \S\ref{sec:dirichletprelims}). Specifically, we compute $1$-and $2$-level statistics of its zeros in what we call \emph{the global regime} (we discuss this in greater detail shortly), and then show how they imply statistics (such as the $n$-level densities) in the local regime agree with random matrix theory predictions.  At the cost of significantly more laborious calculations, one could calculate higher order statistics using our techniques.

In this introduction we only state the results in the global regime; we save for later sections the full statements of the local results, as these require further notation (which many readers are likely to be familiar with) to state, and content ourselves with remarking that we can prove agreement with random matrix theory predictions. Whereas the Hughes-Rudnick results were restricted by the support of $\hat{\phi}$, our global test function $\psi$ is periodic and our results are only restricted by a decay condition on the Fourier coefficients $\hat{\psi}(n)$.

In what follows, let $Q \in \F_q[T]$ be of degree $d \geq 2$.
Let $\fcal_Q$ be the family of primitive Dirichlet characters $\chi : \F_q[T] \to \C$ of modulus $Q$, and let $\fcal_Q^\even$ be the subfamily of even characters in $\fcal_Q$.

%%%%%%%%%%%%%%%%%%%%%%%%%%%%%%%%%%%%%
%%%%%%%%%%%%%%%%%%%%%%%%%%%%%%%%%%%%%
\subsubsection{$1$-Level Statistics}

A \emph{($1$-dimensional) test function of period $1$} is a holomorphic Fourier series $\psi(s) = \sum_{n \in \Z} \hat{\psi}(n)e(ns)$.
The average or expectation of a function $F: \fcal_Q \to \C$ is
\begin{align}
\ebb F \ = \  \frac{1}{\# \fcal_Q}\sum_{\chi \in \fcal_Q} F(\chi);
\end{align}
this sum is well-defined as there are only finitely many $\chi \in \fcal_Q$. With these definitions, we set
\begin{align}\label{eq_Fsub1chi}
F_{1, \chi}(\psi)
&\ \dfeq\ 			\frac{1}{d - 1}
					\sum_{-\frac{T_q}{2} \le \gamma_\chi < \frac{T_q}{2}}
					\psi\pa{\frac{\gamma_\chi}{T_q}},
\end{align}					
where $\gamma_\chi$ runs through the ordinates of the zeros $1/2 + i\gamma_\chi$ of $L(s, \chi)$, and
\begin{equation}
T_q \ := \ \frac{2\pi}{\log q}.
\end{equation}
%which is the correct rescaling for zeros near the central point. With these definitions, we set
Note that \eqref{eq_Fsub1chi} is \emph{not} the standard 1-level density as we are not taking a Schwartz test function of rapid decay. Indeed \eqref{eq_Fsub1chi} is essentialy a Weyl sum. We refer to this and related quantities as \emph{global statistics}, and say we are investigating the zeros in the \emph{global regime}. This is in contrast to the \emph{local statistics} in the \emph{local regime} (which are the zeros near the central point). We briefly explain our choice of notation. We choose to refer to quantities such as \eqref{eq_Fsub1chi} as global statistics as all the zeros can contribute on the same order. Notice, though, that this is not the same as other statistics such as the density of states. The reason is that we are renormalizing the zeros \emph{not} by the correct global quantity (which would be the average spacing between all $d$ zeros), but rather by the correct scaling factor for zeros near the central point. Thus this statistic is a bit of a hybrid, and allows for the exploration of certain statistics with $q$ growing with the degree.

The reason we choose to study these quantities is that we are able to immediately pass from a determination of these global statistics to the more standard local statistics such as the $1$-level density; in particular, we show that the $1$-level density of our family agrees with the scaling limit of unitary matrices as $d\to\infty$ and isolate lower order terms. See \S\ref{sec:onelevelglobal} for the full statement and proof.

\begin{thm}\label{thm:1levelexp}
Suppose $Q$ is irreducible of degree $d \geq 2$.
Let $\psi$ be a test function of period $1$ such that
\begin{align}\label{eq_4.4}
					C(\psi)
&\ = \ 				\sum_{n \in \Z}
					|\hat{\psi}(n)|q^{|n|/2}
\end{align}
converges.
Then
\begin{align}\label{eq_4.5}
					\ebb F_{1, \chi}(\psi)
&\ = \ 				\hat{\psi}(0)
					-		\frac{1}{(d - 1)(q - 1)}
							\sum_{n \in \Z}
							\frac{\hat{\psi}(n)}{q^{|n|/2}}
					+		O\pa{
							\frac{C(\psi)}{dq^d}}.
\end{align}
\end{thm}

\begin{rem}
Note that the theorem above proves that the eigenangles are uniformly distributed $\bmod \ 1$
\end{rem}

The variance of a function $F : \fcal_Q \to \C$ is
\begin{equation}
					\Var F
\ := \ 				\ebb|F - \ebb F|^2
	\ = \ 				\ebb |F|^2 - |\ebb F|^2.
					\nonumber
\end{equation}
Our second result concerns the variance of the $F_{1,\chi}$.

\begin{thm}\label{thm:1levelvar}
Suppose $Q$ is irreducible of degree $d \geq 2$.
Let $\psi$ be a test function of period $1$ such that $C(\psi)$ converges, where $C(\psi)$ is defined in Theorem \ref{thm:1levelexp}.
Then
\begin{align}
					\Var F_{1, \chi}(\psi)
&\ = \ 				\frac{1}{(d - 1)^2}
					\sum_{n \in \Z}
					|n||\hat{\psi}(n)|^2
					+		O\pa{
							\frac{C(\psi)^2}{d^2q^d}
							}.
\end{align}
\end{thm}

%%%%%%%%%%%%%%%%%%%%%%%%%%%%%%%%%%%%%
%%%%%%%%%%%%%%%%%%%%%%%%%%%%%%%%%%%%%
\subsubsection{$2$-Level Statistics}\label{subsec:2levelstatinintro}

A \emph{$2$-dimensional test function of period $1$} is a bivariate Fourier series $\psi(s_1, s_2) = \psi_1(s_1)\psi_2(s_2) =  \sum_{n_1, n_2 \in \Z} \hat{\psi}_1(n_1)\hat{\psi}_2(n_2)e(n_1s_1 + n_2s_2)$.
We set
\begin{align}\label{eq_Fsub2chi}
					F_{2,\chi}(\psi)
\ = \ 					\frac{1}{(d - 1)^2}
					\sum_{\substack{-\frac{T_q}{2} \le \gamma_{\chi, 1}, \gamma_{\chi, 2} < \frac{T_q}{2} \\ \gamma_{\chi, 2} \neq \gamma_{\chi, 1}}}
					\psi\pa{\frac{\gamma_{\chi, 1}}{T_q}, \frac{\gamma_{\chi, 2}}{T_q}}.
\end{align}

Our third global result is the 2-level analogue of Theorem \ref{thm:1levelexp}, which will again confirm agreement with unitary symmetry. See \S\ref{sec:twolevelglobal} for the proof. As the effort of retaining lower-order terms becomes laborious here, and further these terms are not needed to determine the symmetry type, we do not attempt to compute an analogue of Theorem \ref{thm:1levelvar} or other results related to higher level moments.

\begin{thm}\label{thm:2levelexp}
Suppose $Q$ is irreducible of degree $d \geq 2$.
Let $\psi$ be a $2$-dimensional test function of period $1$ such that
\begin{align}
					C(\psi)
&\ = \ 				C(\psi_1)C(\psi_2)
\end{align}
converges, where $C(\psi)$ is defined for $1$-dimensional test functions $\phi$ of period $1$ in Theorem \ref{thm:1levelexp}.
Let $\psi_\diag(s) = \psi(s, s)$.
Then
\begin{align}
					\ebb F_{2, \chi}(\psi)
&\ = \ 				{-\ebb F_{1,\chi}(\psi_\diag) }
					+		\hat{\psi}(0, 0)
					+		\frac{1}{(d - 1)^2}
							\sum_{n \in \Z}
							|n|
							\hat{\psi}(n, -n)
					+		\frac{C_{2, \Gamma}(\psi)}{q - 1}
							\\
&\qquad
					+ 		O\pa{\frac{C(\psi_1) + C(\psi_2)}{d q^d}}
					+		O\pa{\frac{C(\psi)}{d^2 q^d}},
					\nonumber
\end{align}
where
\begin{align}
C_{2, \Gamma}(\psi)
\ = \  					{-\frac{1}{d - 1}}
					\sum_{n \in \Z}
					\frac{\hat{\psi}(0, n) + \hat{\psi}(n, 0)}{q^{|n|/2}}
					+
					\frac{1}{(d - 1)^2}
					\sum_{n_1, n_2 \in \Z}
					\frac{\hat{\psi}(n_1, n_2)}{q^{(|n_1| + |n_2|)/2}}.
\end{align}
\end{thm}

%%%%%%%%%%%%%%%%%%%%%%%%%%%%%%%%%%%%%%
%%%%%%%%%%%%%%%%%%%%%%%%%%%%%%%%%%%%%%
\subsubsection{Local Regime}

In Section \ref{sec:thelocalregime} we explain how we can easily pass from the global to the local regime. Our global statements above imply local statements analogous to those of \cite{HR}---namely, the 1-level expectation, 1-level variance, and 2-level expectation at the central point--- agree with unitary symmetry. Instead of looking at moments we could compute the $1$- and $2$-level densities, and again obtain agreement only with unitary symmetry. Our results are more general and also include lower order terms. There are now several procedures to compute these lower order terms in many random matrix ensembles and systems in mathematical physics; see for example \cite{CFZ1,CFZ2,CS,GHK} for some methods, and \cite{Mil2} for an example where different families of elliptic curve $L$-functions have the same main term but different lower order terms due to differences in their arithmetic.

In Section \ref{sec:montgomeryhyp} we discuss  an $\F_q(T)$-analogue of Montgomery's Hypothesis about the distribution of primes in arithmetic progressions, which Fiorilli and Miller \cite{FioM} show would remove the restriction on the Hughes-Rudnick results. For additional examples on the interplay between conjectures on the distribution of primes and comparisons between zeros of $L$-functions, eigenvalues of random matrix ensembles and energy levels of heavy nuclei, see \cite{BerKea,Kea}.

%\subsection*{Acknowledgments} The authors would like to thank Jon Keating and Ze\'{e}v Rudnick for helpful comments on an earlier draft.

%%%%%%%%%%%%%%%%%%%%%%%%%%%%%%%%%%%%%%%%%%%%%%%%%%%%%%%%%%%%%%%%%%%%%%%%%%%%%%%%%%%%%%%%%%%%%%%%%%%%%%%%%%%%%%%%%%%%%%%%%%%%%
%%%%%%%%%%%%%%%%%%%%%%%%%%%%%%%%%%%%%%%%%%%%%%%%%%%%%%%%%%%%%%%%%%%%%%%%%%%%%%%%%%%%%%%%%%%%%%%%%%%%%%%%%%%%%%%%%%%%%%%%%%%%%
%%%%%%%%%%%%%%%%%%%%%%%%%%%%%%%%%%%%%%%%%%%%%%%%%%%%%%%%%%%%%%%%%%%%%%%%%%%%%%%%%%%%%%%%%%%%%%%%%%%%%%%%%%%%%%%%%%%%%%%%%%%%%
\section{Dirichlet $L$-Function Preliminaries}\label{sec:dirichletprelims}

We always assume $q$ is a prime power below.
We write $\sum', \prod'$  to denote a sum or product restricted to monic polynomials in $\F_q[T]$, and $\sum_P, \prod_P$ to denote a sum or product over irreducibles in $\F_q[T]$.
If $f \in \F_q[T]$, then $|f|$ equals $0$ if $f = 0$ and $q^{\deg f}$ if $f \neq 0$.

Fix a nonconstant modulus $Q \in \F_q[T]$ of degree $d$, and consider Dirichlet characters $\chi : \F_q[T] \to \C$ of modulus $Q$.
To each nontrivial character $\chi$, one associates the $L$-function
\begin{align}\label{eq_2.1}
					L(s, \chi)
&\ \dfeq\ 			{\sideset{}{^{'}}\sum_f}
					\frac{\chi(f)}{|f|^s}
\ =\ 					\sum_{n = 0}^{d - 1}
					{\sideset{}{^{'}}\sum_{\deg f = n}} \chi(f)
					q^{-ns}.
\end{align}
We briefly review its properties, following Chapter 4 of \cite{Ro}.
It possesses the Euler product
\begin{align}									%%%	Some compiler glitch with label 2.2.
					L(s, \chi)
\ =\ 				 	{\sideset{}{^{'}}\prod_P}
					\frac{1}{1 - \chi(P)|P|^{-s}}.
\end{align}
Taking logarithmic derivatives of both sides gives
\begin{align}\label{eq_2.3}
					\frac{L'}{L}(s, \chi)	
\ =\  					-(\log q)
					\sum_{n = 0}^\infty
					c_\chi(n)
					q^{-ns},
\end{align}
where $c_\chi(n) = {\sum_{\deg f = n}^{'}} \Lambda(f)\chi(f)$ and
\begin{align}\label{eq_2.4}
					\Lambda(f)
\ =\					\left\{
					\begin{array}{ll}
					\deg P		&\text{$f = P^\nu$ for some irreducible monic $P$ and $\nu \in \Z_+$}
									\\
					0				&\text{otherwise}
					\end{array}
					\right.
\end{align}
is the von Mangoldt function over $\F_q[T]$.

Since we wish to emphasize the analogy between these $L$-functions and number-field Dirichlet $L$-functions, we prefer to consider their zeros in the variable $s$ rather than $q^{-s}$. The Riemann Hypothesis, proved for these $L$-functions by Weil \cite{We2}, implies that the critical zeros of $L(s, \chi)$ live on the line $\Re s = 1/2$ and thus are \emph{vertically periodic} with period $2\pi/\log q$. Moreover, the Riemann Hypothesis implies that $c_\chi(n) \ll dq^{n/2}$ for all $\chi\neq \chi_0$  (see \cite{Ro}).

%%%%%%%%%%%%%%%%%%%%%%%%%%%%%%%%%%%%%%%%%%%%%%%%%%%%%%%%%%%%%%%
%%%%%%%%%%%%%%%%%%%%%%%%%%%%%%%%%%%%%%%%%%%%%%%%%%%%%%%%%%%%%%%
%%%%%%%%%%%%%%%%%%%%%%%%%%%%%%%%%%%%%%%%%%%%%%%%%%%%%%%%%%%%%%%
%\subsection{Completed $L$-Functions}

We consider the completed $L$-function (a good reference is Chapter 7 of \cite{We1}).
Suppose $\chi$ is primitive.
Then the completed $L$-function associated to $\chi$ is
\begin{align}\label{eq_2.5}
					\lcal(s, \chi)
\ = \  					\frac{1}{1 - \lambda_\infty(\chi)q^{-s}}
					L(s, \chi),
\end{align}
where $\lambda_\infty(\chi)$ equals $1$ if $\chi$ is even, meaning $\F_q^\times \subeq \ker \chi$, and $0$ if $\chi$ is odd.
The functional equation of $\lcal(s, \chi)$ is
\begin{align}\label{eq_2.6}
					\lcal(s, \chi)
\ = \ 					\epsilon(\chi)(q^{d(\chi)})^{1/2 - s}
					\lcal(1 - s, \bar{\chi}),
\end{align}
where $d(\chi) = d - 1 - \lambda_\infty(\chi)$ is the degree of $L(s, \chi)$ seen as a polynomial in the variable $q^{-s}$ and $\epsilon(\chi) \in S^1$ is some root number.
Translating (\ref{eq_2.5})-(\ref{eq_2.6}) into statements about the logarithmic derivatives gives
\begin{align}\label{eq_2.7}
					\dfrac{\lcal'}{\lcal}(s, \chi)
&\ = \     			\dfrac{L'}{L}(s, \chi)
					+		\dfrac{\lambda_\infty(\chi)\log q}{\lambda_\infty(\chi) - q^s}
\end{align}
and
\begin{align}\label{eq_2.8}
					\dfrac{\lcal'}{\lcal}(s, \chi)
&\ = \ 				-d(\chi)\log q
					-		\dfrac{\lcal'}{\lcal}(1 - s, \bar{\chi}).
\end{align}
Therefore, using the fact that $\lambda_\infty(\bar{\chi}) = \lambda_\infty(\chi)$, we find
\begin{align}\label{eq_2.9}
&					-\frac{L'}{L}(1 - s, \bar{\chi})
					\\
&\ = \ 				d(\chi) \log q
					+ 		\frac{L'}{L}(s, \chi)
					+ 		\lambda_\infty(\chi)
							\left(\dfrac{1}{\lambda_\infty(\chi) - q^s} + \dfrac{1}{\lambda_\infty(\chi) - q^{1 - s}}\right)\log q
							\nn
&\ = \ 				d(\chi) \log q
					+ 		\frac{L'}{L}(s, \chi)
					+ 		\lambda_\infty(\chi)
							\left(\dfrac{1}{1 - q^s} + \dfrac{1}{1 - q^{1 - s}}\right)\log q.			
							\nonumber
\end{align}

The following formula is essentially Lemma 2.2 of \cite{FR}.
We re-derive it in Appendix \ref{sec:proofprop3point1} in a way that facilitates comparison with classical Explicit Formulae, such as that of \cite{RS}.
To state the result, abbreviate
\begin{align}\label{eq_3.1}
T_q \ = \  \frac{2\pi}{\log q}.
\end{align} %which is the correct rescaling for zeros near the central point.

\begin{prop}[Explicit Formula]\label{prop:3.1}
Let $Q \in \F_q[T]$ be of degree $d \geq 2$, and let $\chi$ be a nontrivial Dirichlet character of modulus $Q$.
Let $\psi$ be a test function of period $1$.
Then
\begin{align}\label{eq_3.2}
					F_{1, \chi}(\psi)
&\ \dfeq\ 			\frac{1}{d - 1}
					\sum_{-\frac{T_q}{2} \le \gamma_\chi < \frac{T_q}{2}}
					\psi\pa{\frac{\gamma_\chi}{T_q}}
					\\
&\ = \ 				\hat{\psi}(0)
					-		\frac{\lambda_\infty(\chi)}{d - 1}
							\sum_{n \in \Z}
							\frac{\hat{\psi}(n)}{q^{|n|/2}}
					-		\frac{1}{d - 1}			
							\sum_{n = 0}^\infty	
							\frac{c_\chi(n)\hat{\psi}(n) + c_{\bar{\chi}}(n)\hat{\psi}(-n)}{q^{n/2}},
					\nonumber
\end{align}
where $\lambda_\infty(\chi)$ equals $1$ if $\chi$ is even and $0$ otherwise, and $c_\chi(n) = {\sum_{\deg f = n}^{'}} \Lambda(f)\chi(f)$.
\end{prop}

%%%%%%%%%%%%%%%%%%%%%%%%%%%%%%%%%%%%%%%%%%%%%%%%%%%%%%%%%%%%%%%%%%%%%%%%%%%%%%%%%%%%%%%%%%%%%%%%%%%%%%%%%%%%%%%%%%%%%%%%%%%%%
%%%%%%%%%%%%%%%%%%%%%%%%%%%%%%%%%%%%%%%%%%%%%%%%%%%%%%%%%%%%%%%%%%%%%%%%%%%%%%%%%%%%%%%%%%%%%%%%%%%%%%%%%%%%%%%%%%%%%%%%%%%%%
%%%%%%%%%%%%%%%%%%%%%%%%%%%%%%%%%%%%%%%%%%%%%%%%%%%%%%%%%%%%%%%%%%%%%%%%%%%%%%%%%%%%%%%%%%%%%%%%%%%%%%%%%%%%%%%%%%%%%%%%%%%%%
\section{The $1$-Level Global Regime}\label{sec:onelevelglobal}

The computations in this section are closely based on those in \cite{HR}, after which we can easily prove corresponding results about the 1-level density and similar quantities.

\subsection{Expectation}

%%%%%%%%%%%%%%%%%%%%%%%%%%%%%%%%%%%%%%%%%%%%%%%%%%%%%%%%%%%%%%%
%%%%%%%%%%%%%%%%%%%%%%%%%%%%%%%%%%%%%%%%%%%%%%%%%%%%%%%%%%%%%%%
%%%%%%%%%%%%%%%%%%%%%%%%%%%%%%%%%%%%%%%%%%%%%%%%%%%%%%%%%%%%%%%

\begin{proof}[Proof of Theorem \ref{thm:1levelexp}]
By the Explicit Formula
\begin{align}\label{eq_4.2}
					\ebb F_{1, \chi}(\psi)
\ = \ 					\hat{\psi}(0)
					-		\frac{1}{d - 1}
							\cdot \frac{\# \fcal_Q^\even}{\# \fcal_Q}
							\sum_{n \in \Z}
							\frac{\hat{\psi}(n)}{q^{|n|/2}}
					+		 \ebb (F_{1, \chi}(\psi)^\osc),
\end{align}
where
\begin{align}\label{eq_4.3}
					F_{1, \chi}(\psi)^\osc
\ = \ 					-\frac{1}{d - 1}			
					\sum_{n = 0}^\infty	
					\frac{c_\chi(n)\hat{\psi}(n) + c_{\bar{\chi}}(n)\hat{\psi}(-n)}{q^{n/2}}.
\end{align}
Since $Q$ is monic and irreducible, the only imprimitive character modulo $Q$ is the principal character $\chi_0$.
In this case, there are $(|Q| - 1)/(q - 1)$ even characters including $\chi_0$, so we know $\#\fcal_Q^\even /\#\fcal_Q$ is roughly $1/(q - 1)$.
It remains to estimate $\ebb (F_{1, \chi}(\psi)^\osc)$. %%%

By Schur orthogonality,
\begin{align}\label{eq_4.6}
					\ebb \chi(f)
\ = \ 					\left\{
					\begin{array}{ll}	
					0						&f\equiv 			0\pmod{Q}\\					
					1						&f \equiv			1\pmod{Q}\\
					-1/\#\fcal_Q	&\text{otherwise}
					\end{array}
					\right.
\end{align}
for all $f \in \F_q[T]$, and similarly with $\ebb \bar{\chi}(f)$.
Therefore
\begin{align}\label{eq_4.7}
&					\ebb(F_{1, \chi}(\psi)^\osc)
					\\
&\ = \ 				-\frac{1}{d - 1}
					\sum_{n = 0}^\infty
					\left(	
					{\sideset{}{^{'}}\sum_{\substack{\deg f = n \\ f \equiv 1\pmod{Q}}}}
					-		\frac{1}{\#\fcal_Q}
							{\sideset{}{^{'}}\sum_{\substack{\deg f = n \\  f \not\equiv 0,1\pmod{Q}}}}
					\right)
					\Lambda(f)
					\frac{\hat{\psi}(n) + \hat{\psi}(-n)}{q^{n/2}}.
					\nonumber
\end{align}
To estimate the contribution of the first term in the big parenthesized expression above we make use of the function-field analogue of the Brun-Titchmarsh Theorem (see \cite{Hsu}), which states that
\begin{align}
					{\sideset{}{^{'}}\sum_{\substack{\deg f = n \\ f\equiv 1\pmod{Q}}}}
					\Lambda(f)
&\ \le \  			Cq^{n - d}
\end{align}
for some $C > 0$ independent of $Q, n$.
Thus the contribution from this term is
\begin{align}
					-\frac{1}{(d - 1)}
					\sum_{n = 0}^\infty
					{\sideset{}{^{'}}\sum_{\substack{\deg f = n \\ f\equiv 1\pmod{Q}}}}
					\Lambda(f)
					\frac{\hat{\psi}(n) + \hat{\psi}(-n)}{q^{n/2}}					
&\ \ll \ 				\frac{1}{dq^d}
					\sum_{n \in \Z}
					\hat{\psi}(n)
					q^{|n|/2}.
\end{align}
On the other hand, by the Prime Number Theorem in this setting (see \cite{Ro}) we  have
\begin{align}
					{\sideset{}{^{'}}\sum_{\deg f = n}}
					\Lambda(f)
\ = \ 					q^n + O(q^{n/2}),
\end{align}
where the implied constant is independent of $q$.
Thus the second term of the big parenthesized expression in (\ref{eq_4.7}) contributes with the same order, completing the proof.
\end{proof}

\begin{cor}
Suppose $Q$ is irreducible.
Let $\psi$ be a test function of period $1$ such that
\begin{align}\label{eq_4.11}
					\hat{\psi}(n)
&\ \ll \ 				\frac{1}{|n|^{1 + \epsilon}q^{|n|/2}}
\end{align}
for some $\epsilon > 0$.
Then
\begin{align}\label{eq_4.12}
					\ebb F_{1, \chi}(\psi)
&\ = \ 				\hat{\psi}(0)
					-		\frac{1}{(d - 1)(q - 1)}
							\sum_{n \in \Z}
							\frac{\hat{\psi}(n)}{q^{|n|/2}}
					+		O\pa{\frac{1}{d}},
\end{align}
where $F_{1, \chi}$ is defined in \eqref{eq_Fsub1chi}.
\end{cor}

%%%%%%%%%%%%%%%%%%%%%%%%%%%%%%%%%%%%%%%%%%%%%%%%%%%%%%%%%%%%%%%
%%%%%%%%%%%%%%%%%%%%%%%%%%%%%%%%%%%%%%%%%%%%%%%%%%%%%%%%%%%%%%%
%%%%%%%%%%%%%%%%%%%%%%%%%%%%%%%%%%%%%%%%%%%%%%%%%%%%%%%%%%%%%%%
\subsection{Variance}

\begin{proof}[Proof of Theorem \ref{thm:1levelvar}]
Let $C_{1, \Gamma}(\psi) = (d - 1)^{-1}\sum_{n \in \Z} \hat{\psi}(n)q^{-|n|/2}$.
Then
\begin{align}
				F_{1, \chi}(\psi) - \ebb F_{1, \chi}(\psi)
\ = \ 				(F_{1, \chi}(\psi)^\osc - \ebb (F_{1, \chi}(\psi))^\osc)
				+ 		C_{1, \Gamma}(\psi)\pa{\lambda_\infty(\chi) - \frac{1}{q - 1}},
\end{align}
from which
\begin{align}
&				\Var F_{1, \chi}(\psi)
				\\
&\ = \ 			\Var (F_{1, \chi}(\psi)^\osc)
				+		2	\Re \ebb\Big(
						\pa{F_{1, \chi}(\psi)^\osc - \ebb (F_{1, \chi}(\psi))^\osc}
						C_{1, \Gamma}(\psi)
						\lambda_\infty(\chi)\Big)
							+		O\pa{\frac{1}{d^2}}
						\nn
&\ = \ 			\Var (F_{1, \chi}(\psi)^\osc)
				+		O\pa{\frac{1}{d} \ebb(F_{1, \chi}(\psi)^\osc)}.
				\nonumber
\end{align}
Next, $\Var (F_{1, \chi}(\psi)^\osc) = \ebb |F_{1, \chi}(\psi)^\osc|^2 - |\ebb (F_{1, \chi}(\psi)^\osc)|^2$, where
\begin{align}
				|F_{1, \chi}(\psi)^\osc|^2
&\ = \ 			\frac{1}{(d - 1)^2}
				\sum_{n_1, n_2 = 0}^\infty
				{\sideset{}{^{'}}\sum_{\substack{\deg f_1 = n_1 \\ \deg f_2 = n_2}}}
				\frac{\Lambda(f_1)\Lambda(f_2)}{q^{(n_1 + n_2)/2}}
				\\
&\qquad	\Bigg(
				\chi(f_1)\bar{\chi}(f_2)	\hat{\psi}(n_1)\bar{\hat{\psi}}(n_2)
				\ +\		\chi(f_1)\chi(f_2) 						\hat{\psi}(n_1)\bar{\hat{\psi}}(-n_2)
				%\right.
				\nn
&\qquad	%\left.
				\ \ \ \ +\		\bar{\chi}(f_1)\bar{\chi}(f_2) 	\hat{\psi}(-n_1)\bar{\hat{\psi}}(n_2)
				+\		\bar{\chi}(f_1)\chi(f_2) 				\hat{\psi}(-n_1)\bar{\hat{\psi}}(-n_2)
				\Bigg).
				\nonumber
\end{align}
Again by Schur orthogonality,
\begin{align}
				\ebb
				\pa{\chi(f_1)\bar{\chi}(f_2)}
\ = \ 				\left\{
				\begin{array}{ll}					
				0										&\text{$f_1 \equiv 0$ or $f_2 \equiv 0\pmod{Q}$}\\
				1										&f_1 \equiv f_2 \not\equiv 0 \pmod{Q}\\
				-1/\#\fcal_Q					&\text{otherwise}
				\end{array}
				\right.
				\\
				\ebb
				\pa{\chi(f_1)\chi(f_2)}
\ = \ 				\left\{
				\begin{array}{ll}					
				0										&\text{$f_1 \equiv 0$ or $f_2 \equiv 0\pmod{Q}$}\\
				1										&f_1f_2 \equiv 1 \pmod{Q}\\
				-1/\#\fcal_Q					&\text{otherwise.}
				\end{array}
				\right.
\end{align}
Therefore
\begin{align}\label{eq_4.20}
&				\ebb |F_{1, \chi}(\psi)^\osc|^2
				\\
&\ = \ 			\frac{1}{(d - 1)^2}
				\sum_{n_1, n_2 = 0}^\infty
				\frac{1}{q^{(n_1 + n_2)/2}}
				\left(C_1(n_1, n_2; Q)
						\pa{\hat{\psi}(n_1) \bar{\hat{\psi}}(n_2) + \hat{\psi}(-n_1) \bar{\hat{\psi}}(-n_2)}
				\right.
				\nn
&\qquad	+		\left.
				 		C_2(n_1, n_2; Q)
						\pa{\hat{\psi}(n_1) \bar{\hat{\psi}}(-n_2) + \hat{\psi}(-n_1) \bar{\hat{\psi}}(n_2)}
				\right)
				\nn
&\qquad	+ 		O\pa{\frac{1}{\# \fcal_Q}
								\pa{\frac{1}{d - 1}
								\sum_{n = 0}^\infty
								{\sideset{}{^{'}}\sum_{\deg f = n}}
								\Lambda(f) \frac{\hat{\psi}(n)}{q^{n/2}}}^2
						},
				\nonumber
\end{align}
where
\begin{align}
				\label{eq_4.21}
				C_1(n_1, n_2; Q)
&\ = \ 			{\sideset{}{^{'}}\sum_{\substack{\deg f_1 = n_1 \\ \deg f_2 = n_2 \\ f_1 \equiv f_2 \not\equiv 0 \pmod{Q}}}}
				\Lambda(f_1)\Lambda(f_2),
				\\
				\label{eq_4.22}
				C_2(n_1, n_2; Q)
&\ = \ 			{\sideset{}{^{'}}\sum_{\substack{\deg f_1 = n_1 \\ \deg f_2 = n_2 \\ f_1 f_2 \equiv 1 \pmod{Q}}}}
				\Lambda(f_1)\Lambda(f_2),
\end{align}
and the big-$O$ term of (\ref{eq_4.20}) is $O(C(\psi)^2/(d^2 q^d))$.

As before, we use the Hsu-Brun-Titchmarsh Theorem to bound the contribution of the $C_2$ sum:
\begin{align}
&				\frac{1}{(d - 1)^2}
				\sum_{n_1, n_2 = 0}^\infty
				C_2(n_1, n_2; Q)				
				\frac{\hat{\psi}(n_1) \bar{\hat{\psi}}(-n_2) + \hat{\psi}(-n_1) \bar{\hat{\psi}}(n_2)}{q^{(n_1 + n_2)/2}}
				\\
&\ \ll \ 			\frac{1}{(d - 1)^2q^d}
				\sum_{n_1, n_2 = 0}^\infty
				\pa{\hat{\psi}(n_1) \bar{\hat{\psi}}(-n_2) + \hat{\psi}(-n_1) \bar{\hat{\psi}}(n_2)}
				q^{(n_1 + n_2)/2}
				\nn
&\ \ll \ 			\frac{C(\psi)^2}{d^2 q^d}.
				\nonumber
\end{align}
Using the same theorem, we break the contribution of the $C_1$ sum into a main diagonal term and an off-diagonal term that depends on $Q$; the latter contributes with the same order as the $C_2$ sum:
\begin{align}
&				\frac{1}{(d - 1)^2}
				\sum_{n_1, n_2 = 0}^\infty
				C_1(n_1, n_2; Q)
				\frac{\hat{\psi}(n_1) \bar{\hat{\psi}}(n_2) + \hat{\psi}(-n_1) \bar{\hat{\psi}}(-n_2)}{q^{(n_1 + n_2)/2}}
				\\
&\ = \ 			\frac{1}{(d - 1)^2}
				\sum_{n = 0}^\infty
				{\sideset{}{^{'}}\sum_{\deg f = n}}
				\Lambda(f)^2
				\frac{|\hat{\psi}(n)|^2 + |\hat{\psi}(-n)|^2}{q^n}
				+		O\pa{\frac{C(\psi)^2}{d^2 q^d}}.
				\nonumber
\end{align}
By writing
\begin{align}
				{\sideset{}{^{'}}\sum_{\deg f = n}} \Lambda(f)^2 \ = \  n {\sideset{}{^{'}}\sum_{\deg f = n}} \Lambda(f) \ = \  n(q^n + O(q^{n/2}))
\end{align}
we conclude the proof.
\end{proof}

\begin{cor}
Suppose $Q$ is irreducible.
Let $\psi$ be a test function of period $1$ such that (\ref{eq_4.11}) holds.
Then
\begin{align}
					\Var F_{1, \chi}(\psi)
&\ = \ 				\frac{1}{(d - 1)^2}
					\sum_{n \in \Z}
					|n||\hat{\psi}(n)|^2
					+		O\pa{\frac{1}{d^2}},
\end{align}
where $F_{1, \chi}$ is defined in \eqref{eq_Fsub1chi}.
\end{cor}

%%%%%%%%%%%%%%%%%%%%%%%%%%%%%%%%%%%%%%%%%%%%%%%%%%%%%%%%%%%%%%%%%%%%%%%%%%%%%%%%%%%%%%%%%%%%%%%%%%%%%%%%%%%%%%%%%%%%%%%%%%%%%
%%%%%%%%%%%%%%%%%%%%%%%%%%%%%%%%%%%%%%%%%%%%%%%%%%%%%%%%%%%%%%%%%%%%%%%%%%%%%%%%%%%%%%%%%%%%%%%%%%%%%%%%%%%%%%%%%%%%%%%%%%%%%
%%%%%%%%%%%%%%%%%%%%%%%%%%%%%%%%%%%%%%%%%%%%%%%%%%%%%%%%%%%%%%%%%%%%%%%%%%%%%%%%%%%%%%%%%%%%%%%%%%%%%%%%%%%%%%%%%%%%%%%%%%%%%
\section{The $2$-Level Global Regime}\label{sec:twolevelglobal}

Since the computations rapidly become laborious, we do not derive an unaveraged $2$-level explicit formula, but instead compute the expectation of $F_{2, \chi}(\psi)$ directly. %We do not compute the higher level moments for the same reason.

\begin{proof}[Proof of Theorem \ref{thm:2levelexp}]
Let
\begin{align}
					\psi^\ast(s_1, s_2)
\ = \  					\frac{1}{(d - 1)^2}
					\psi\pa{\frac{s_1}{T_q}, \frac{s_2}{T_q}}.
\end{align}
Let $\ell_c$ be defined as in the proof of Proposition \ref{prop:3.1}.
For $j = 1,2$, let $c_j = 1/2 + \epsilon_j$, where $0 < \epsilon_1 < \epsilon_2 < 1/2$.
Writing $\int_{\ccal_j} = \int_{\ell_{c_j}} - \int_{\ell_{1 - c_j}}$, Cauchy's Theorem implies
\begin{align}\label{eq_5.5}
					F_{2, \chi}(\psi)
&\ = \ 				-		F_{1, \chi}(\psi_\diag)
					+		F_{2, \chi}(\psi; \epsilon_1, \epsilon_2)
					- 		F_{2, \chi}(\psi; \epsilon_1, -\epsilon_2)
					-		F_{2, \chi}(\psi; -\epsilon_1, \epsilon_2)
					\\
&\qquad		
					+		F_{2, \chi}(\psi; -\epsilon_1, -\epsilon_2)
					+ 		O(\max(\epsilon_1, \epsilon_2)),
					\nonumber
\end{align}
where
\begin{align}
&					F_{2, \chi}(\psi; \epsilon_1, \epsilon_2) 	
					\\
&\ = \ 				\frac{1}{(2\pi)^2}
					\iint_{A_q}
					\frac{L'}{L}(1/2 + \epsilon_1 + it_1, \chi) \frac{L'}{L}(1/2 + \epsilon_2 + it_2, \chi)
					\psi^\ast((t_j -i\epsilon_j)_{j = 1,2})\,\ud t_1\,\ud t_2
					\nonumber
\end{align}
and $A_q = [-T_q/2, +T_q/2]^2$.

Again, we employ the functional equation to replace those terms of (\ref{eq_5.5}) that have $-\epsilon_1$ or $-\epsilon_2$ as a parameter.
First, define
\begin{align}
&					F_{2, \chi}^{2,2}(+1)
					\\
&\ = \ 				\frac{1}{(2\pi)^2}
					\iint_{A_q}
							\left(
							\frac{L'}{L}(1/2 + \epsilon_1 + it_1, \chi)\frac{L'}{L}(1/2 + \epsilon_2 - it_2, \bar{\chi})
							\psi^\ast(t_1 - i\epsilon_1, t_2 + i\epsilon_2)
							\right.
					\nn
&\qquad		+		\left.
							\frac{L'}{L}(1/2 + \epsilon_2 + it_2, \chi)\frac{L'}{L}(1/2 + \epsilon_1 - it_1, \bar{\chi})
							\psi^\ast(t_1 + i\epsilon_1, t_2 - i\epsilon_2)
							\right)
							\,\ud t_1\,\ud t_2,
					\nn	
&					F_{2, \chi}^{2,2}(-1)
					\\
&\ = \ 				\frac{1}{(2\pi)^2}
					\iint_{A_q}
					\pa{\prod_{j = 1,2}
					\frac{L'}{L}(1/2 + \epsilon_j - it_j, \chi)}
					\psi^\ast(t_1 + i\epsilon_1, t_2 + i\epsilon_2)\, \ud t_1\,\ud t_2,
					\nn
&					F_{2, \chi}^{3,3}
					\\
&\ = \ 				\frac{1}{(2\pi)^2}
					\iint_{A_q}
					G_\chi(1/2 + it_1)G_\chi(1/2 + it_2)
					\psi^\ast(t_1, t_2)\,\ud t_1\,\ud t_2,
					\nonumber
\end{align}
where
\begin{align}
					G_\chi(s)
\ = \ 					\lambda_\infty(\chi)\left(-1 + \dfrac{1}{1 - q^s} + \dfrac{1}{1 - q^{1 - s}}\right)\log q.		
\end{align}
Also define
\begin{align}
&					F_{2, \chi}^{1,2}(\delta)
					\\
&\ = \ 				\frac{1}{(2\pi)^2}
					\iint_{A_q}
					((d - 1) \log q)
							\left(
							\frac{L'}{L}(1/2 + \epsilon_1 + i\delta t_1, \chi)
							\psi^\ast(t_1 - i\delta \epsilon_1, t_1 - i \epsilon_2)
							\right.
					\nn
&\qquad		+ 		\left.
							\frac{L'}{L}(1/2 + \epsilon_2 + i\delta t_2, \chi)
							\psi^\ast(t_1 - i\epsilon_1, t_1 - i\delta \epsilon_2)
							\right)	
							\,\ud t_1\,\ud t_2,
					\nn
&					F_{2, \chi}^{1,3}
					\\
&\ = \ 				\frac{1}{(2\pi)^2}
					\iint_{A_q}
					((d - 1)\log q)
					\left(
					G_\chi(1/2 + it_1)	
					+		G_\chi(1/2 + it_2)
							\right)
					\psi^\ast(t_1, t_2)
					\,\ud t_1\,\ud t_2,
					\nonumber
\end{align}
\begin{align}
&					F_{2, \chi}^{2,3}(\delta)
					\\
&\ = \ 				\frac{1}{(2\pi)^2}
					\iint_{A_q}
					\left(
					\frac{L'}{L}(1/2 + \epsilon_1 + i \delta t_1, \chi)G_\chi(1/2 + \epsilon_2 + i t_2)
					\psi^\ast(t_1 - i\delta \epsilon_1, t_2 - i\epsilon_2)
					\right.
					\nn
&\qquad		+		\left.
							\frac{L'}{L}(1/2 + \epsilon_2 + i\delta t_2, \chi)G_\chi(1/2 + \epsilon_1 + i t_1)
							\psi^\ast(t_1 - i\epsilon_1, t_2 - i\delta\epsilon_2)
							\right)
					\,\ud t_1,\ud t_2.	
					\nonumber			
\end{align}
It is straightforward to check that
\begin{align}
					-(F_{2, \chi}(\psi; \epsilon_1, -\epsilon_2) + F_{2, \chi}(\psi; -\epsilon_1, \epsilon_2))
&\ = \ 				F_{2, \chi}^{1,2}(+1)
					+		F_{2, \chi}^{2,2}(+1)
					+		F_{2, \chi}^{2,3}(+1)
\end{align}
and
\begin{align}
					F_{2, \chi}(\psi; -\epsilon_1, -\epsilon_2)
&\ = \ 				\frac{1}{(2\pi)^2}
					\iint_{A_q}
					((d - 1) \log q)^2
					\psi^\ast(t_1, t_2)\,\ud t_1\,\ud t_2
					\\
&\qquad		+		F_{2, \bar{\chi}}^{2,2}(-1)
					+		F_{2, \chi}^{3,3}	
					+		F_{2, \bar{\chi}}^{1,2}(-1)
					+		F_{2, \chi}^{1,3}
					+		F_{2, \bar{\chi}}^{2,3}(-1)
					\nn
&\ = \ 				\hat{\psi}(0, 0)
					+		F_{2, \bar{\chi}}^{2,2}(-1)
					+		F_{2, \chi}^{3,3}	
					+		F_{2, \bar{\chi}}^{1,2}(-1)
					+		F_{2, \chi}^{1,3}
					+		F_{2, \bar{\chi}}^{2,3}(-1).
					\nonumber
\end{align}
In what follows, we estimate each of the individual contributions.
We will implicitly substitute the appropriate Dirichlet series and send $\epsilon_1, \epsilon_2 \to 0$ in all of the $L'/L$ terms.

\subsubsection*{Diagonal Contributions}

By an argument similar to that in the proof of Proposition \ref{prop:3.1},
\begin{align}
					F_{2,\chi}^{3,3}
&\ = \ 				\frac{\lambda_\infty(\chi)}{(d - 1)^2}
					\sum_{n_1, n_2 \in \Z}
					\frac{\hat{\psi}(n_1, n_2)}{q^{(|n_1| + |n_2|)/2}}.
\end{align}
The analysis of the expectation of the two $F_{2, \chi}^{2,2}$ terms is reminiscent of that of $\ebb |F_{1, \chi}(\psi)^\osc|^2$ in the proof of Theorem \ref{thm:1levelvar}.
With $C_1(n_1, n_2, Q)$ and $C_2(n_1, n_2, Q)$ defined as in (\ref{eq_4.21}) and (\ref{eq_4.22}), respectively,
\begin{align}
&				\ebb(F_{2, \chi}^{2,2}(+1) + F_{2, \chi}^{2,2}(-1))
				\\
&\ = \ 			\frac{1}{(d - 1)^2}
				\sum_{n_1, n_2 = 0}^\infty
				\frac{1}{q^{(n_1 + n_2)/2}}
				\nn
&\qquad	\left(
						C_1(n_1, n_2; Q)
						\pa{\hat{\psi}(n_1, -n_2) + \hat{\psi}(-n_1, n_2)}
						+
				 		{C_2(n_1, n_2; Q)}
						\hat{\psi}(-n_1, -n_2)
				\right)
				\nn
&\ = \ 			\frac{1}{(d - 1)^2}
				\sum_{n = 0}^\infty
				{\sideset{}{^{'}}\sum_{\deg f = n}}
				\Lambda(f)^2
				\frac{\hat{\psi}(n, -n) + \hat{\psi}(-n, n)}{q^n}
				+
				O\pa{
				\frac{C(\psi)}{d^2 q^d}}
				\nn
&\ = \ 			\frac{1}{(d - 1)^2}
				\sum_{n \in \Z}
				|n|
				\hat{\psi}(n, -n)
				+
				O\pa{
				\frac{C(\psi)}{d^2 q^d}}.
				\nonumber
\end{align}

\subsubsection*{Off-Diagonal Contributions}

Similarly to the argument  in the proof of Proposition \ref{prop:3.1},
\begin{align}
					F_{2,\chi}^{1,3}
&\ = \ 				-\frac{\lambda_\infty(\chi)}{d - 1}
					\sum_{n \in \Z}
					\frac{\hat{\psi}(0, n) + \hat{\psi}(n, 0)}{q^{|n|/2}}.
\end{align}
Also, similarly to the proof of Theorem \ref{thm:1levelexp},
\begin{align}
					\ebb(F_{2, \chi}^{1,2}(+1) + F_{2, \chi}^{1,2}(-1))
&\ = \ 				-\frac{1}{d - 1}
					\sum_{n = 0}^\infty
					\left(	
					{\sideset{}{^{'}}\sum_{\substack{\deg f = n \\ f \equiv 1\pmod{Q}}}}
					-		\frac{1}{\#\fcal_Q}
							{\sideset{}{^{'}}\sum_{\substack{\deg f = n \\  f \not\equiv 0,1\pmod{Q}}}}
					\right)
					\frac{\Lambda(f)}{q^{n/2}}
					\\
&\qquad		(\hat{\psi}(n, 0) + \hat{\psi}(-n, 0) + \hat{\psi}(0, n) + \hat{\psi}(0, -n))
					\nn
&\ \ll \ 				\frac{C(\psi_1) + C(\psi_2)}{dq^d}
					\nonumber
\end{align}
and
\begin{align}
					\ebb(F_{2, \chi}^{2,3}(+1) + F_{2, \chi}^{2,3}(-1))
&\ \ll \ 				\frac{C(\psi)}{d^2 q^d}.
\end{align}
\end{proof}

%%%%%%%%%%%%%%%%%%%%%%%%%%%%%%%%%%%%%%%%%%%%%%%%%%%%%%%%%%%%%%%%%%%%%%%%%%%%%%%%%%%%%%%%%%%%%%%%%%%%%%%%%%%%%%%%%%%%%%%%%%%%%
%%%%%%%%%%%%%%%%%%%%%%%%%%%%%%%%%%%%%%%%%%%%%%%%%%%%%%%%%%%%%%%%%%%%%%%%%%%%%%%%%%%%%%%%%%%%%%%%%%%%%%%%%%%%%%%%%%%%%%%%%%%%%
%%%%%%%%%%%%%%%%%%%%%%%%%%%%%%%%%%%%%%%%%%%%%%%%%%%%%%%%%%%%%%%%%%%%%%%%%%%%%%%%%%%%%%%%%%%%%%%%%%%%%%%%%%%%%%%%%%%%%%%%%%%%%
\section{The Local Regime}\label{sec:thelocalregime}

%%%%%%%%%%%%%%%%%%%%%%%%%%%%%%%%%%%%%%%%%%%%%%%%%%%%%%%%%%%%%%%5
%%%%%%%%%%%%%%%%%%%%%%%%%%%%%%%%%%%%%%%%%%%%%%%%%%%%%%%%%%%%%%%5
%%%%%%%%%%%%%%%%%%%%%%%%%%%%%%%%%%%%%%%%%%%%%%%%%%%%%%%%%%%%%%%5
\subsection{$n$-Level Density}

For us, an \emph{$n$-dimensional test function of moderate decay} is a smooth function $\phi(s_1,$ $\ldots,$ $s_n)$ $=$ $\phi_1(s_1)$ $\dots$ $\phi_n(s_n)$, defined in a region $U \subeq \C^n$ containing $\R^n$, such that $\phi(s) \ll (1 + |s|)^{-(1 + \delta)}$ for some $\delta > 0$.
The (homogeneous) periodization of $\phi$, scaled by a parameter $N$, is
\begin{align}
				\phi_N(s)
\ = \ 				\sum_{\nu \in \Z^n}	
				\phi
				\pa{N(s + \nu_1), \ldots, N(s + \nu_n)}.
\end{align}
Let $U$ be an $N \times N$ unitary matrix with eigenangles $\theta_1, \ldots, \theta_N$.
Then the \emph{$n$-level density} of the $\theta_j$ with respect to $\phi$ is
\begin{align}
				W_{n, U}(\phi)
\ := \ 				\sum_{\substack{1 \le j_1, \ldots, j_n \le N \\ \text{$j_k$ distinct}}}
				\phi_N
				\pa{\frac{\theta_{j_1}}{2\pi}, \ldots, \frac{\theta_{j_n}}{2\pi}}.
\end{align}
Let $L(s, \chi)$ be a Dirichlet $L$-function over $\F_q(T)$.
Recall that $L(s, \chi)$ has $d + O(1)$ zeros of the form $1/2 + \gamma_{\chi, j}$ in an interval of periodicity $[1/2 - iT_q/2, 1/2 + iT_q/2)$.
Therefore, by analogy, we define the \emph{$n$-level density} of $L(s, \chi)$ with respect to $\phi$ to be
\begin{align}\label{eq_6.3}
				W_{n,\chi}(\phi)
\ = \ 				\sum_{\substack{-\frac{T_q}{2} \le \gamma_{\chi,j_1}, \ldots, \gamma_{\chi, j_n} < \frac{T_q}{2} \\ \text{$j_k$ distinct}}}
				\phi_{d - 1}
				\pa{\frac{\gamma_{\chi, 1}}{T_q}, \ldots, \frac{\gamma_{\chi, n}}{T_q}}.
\end{align}

To obtain the \emph{$n$-level density for the family} (and not just one form) we simply average over all elements of the family.

\begin{rem}
To facilitate comparison with \cite{HR}, we compare the above to the situation over number fields in the $1$-level case.
Let $L(s)$ be a Selberg-class $L$-function with analytic conductor $c > 0$.
Write $1/2 + i\gamma_j$ to denote its $j$\textsuperscript{th} critical zero above the real line, ordered by height.
Then the $1$-level density of $L(s)$ with respect to $\phi$ is
\begin{align}
				W_1(\phi)
\ := \ 				\lim_{T \to \infty}
				\sum_{	0 \le |\gamma_j| < T}
				\phi
				\pa{\gamma_j \frac{\log c}{2\pi}}.
\end{align}
Above, $(\log c)/2\pi$ normalizes the average consecutive spacing between ordinates of zeros near the central point to be $1$ in the limit $T \to \infty$.
\end{rem}

\begin{rem} In the definition of the $n$-level density we sum over all the zeros; however, because of our scaling and the decay of $\phi$ most of the contribution comes from the zeros close to the central point. Notice that the sum is over distinct zeros. Using inclusion-exclusion and the explicit formula, it is easy to consider several related quantities, from the expression in \eqref{eq_6.3} to sums without any distinctness restriction. In the number field setting where the zeros are symmetric about the central point, we write them (assuming GRH) as $1/2 + i \gamma_j$, with $j \in \{\dots, -2, -1, 0, 1, 2, \dots\}$ if the form is odd (if the form is even there is no zero index); in this case the $n$-level density is a sum over zeros $\gamma_{j_1}, \dots, \gamma_{j_n}$ such that $j_i \neq \pm j_\ell$ if $i \neq \ell$. These differences are all minor and readily handled, but lead to different combinatorics and different forms of the final answer. \end{rem}

%%%%%%%%%%%%%%%%%%%%%%%%%%%%%%%%%%%%%%%%%%%%%%%%%%%%%%%%%%%%%%%
%%%%%%%%%%%%%%%%%%%%%%%%%%%%%%%%%%%%%%%%%%%%%%%%%%%%%%%%%%%%%%%
%%%%%%%%%%%%%%%%%%%%%%%%%%%%%%%%%%%%%%%%%%%%%%%%%%%%%%%%%%%%%%%
\subsection{Unitary Predictions}

Let $\UL(N)$ be the group of $N \times N$ unitary matrices under Haar measure.
If $F$ is a function on $\UL(N)$, then the expectation of $F$ is
\begin{align}
				\ebb F(U)
\ = \ 				\int_{\UL(N)}
				F(U)
				\,\ud U.
\end{align}
In \cite{HR}, the authors prove that  if $\phi$ is an \emph{even}\footnote{Since the zeros are symmetric, there is no loss in using even test functions.} $1$-dimensional test function of rapid decay, then
\begin{align}
				\ebb W_U(\phi)
&\ \to \ 			\hat{\phi}(0),
				\\
				\ebb(W_{1,U}(\phi) - \ebb W_{1,U}(\phi))^2
&\ \to \ 			\sigma(\phi)^2
\end{align}
as $N \to \infty$, where
\begin{align}
				\sigma(\phi)^2
\ = \ 				\integral{\infty}
				\min(1, |t|)
				\hat{\phi}(t)^2
				\,\ud t.
\end{align}
If $\hat{\phi}$ is supported in the interval $[-2/m, 2/m]$, then they also obtain
\begin{align}
\ebb(W_{1,U}(\phi) - \ebb W_{1,U}(\phi))^m
\ \to\
\left\{
\begin{array}{ll}
0											&\text{$m$ odd}\\
\dfrac{m!}{2^{m/2}(m/2)!}
\sigma(\phi)^m					&\text{$m$ even;}
\end{array}
\right.
\end{align} using inclusion-exclusion one can pass from $n$\textsuperscript{th} centered moments to $n$-level densities (see for example \cite{HM}).
By \cite{KS1}, if $\phi = \phi_1\cdot \phi_2$ is an even $2$-dimensional test function of rapid decay, then
\begin{align}
				\ebb W_{2, U}(\phi)
\ \to \ 			-
				\hat{\phi}_\diag(0)
				+	
				\hat{\phi}(0, 0)
				+
				\integral{\infty}
				|t|
				\hat{\phi}_1(t)
				\hat{\phi}_2(t)
				\,\ud t,
\end{align}
where $\phi_\diag(s) = \phi(s, s)$.

%%%%%%%%%%%%%%%%%%%%%%%%%%%%%%%%%%%%%%%%%%%%%%%%%%%%%%%%%%%%%%%
%%%%%%%%%%%%%%%%%%%%%%%%%%%%%%%%%%%%%%%%%%%%%%%%%%%%%%%%%%%%%%%
%%%%%%%%%%%%%%%%%%%%%%%%%%%%%%%%%%%%%%%%%%%%%%%%%%%%%%%%%%%%%%%
\subsection{From Global to Local}

Set $\psi = \phi_{d - 1}$ in the results from Sections \ref{sec:onelevelglobal} and \ref{sec:twolevelglobal}.
Recall that $\psi(s_1, \ldots, s_n)$ is a Fourier series. The Fourier \emph{transform} of $\phi$ is
\begin{align}
				\hat{\phi}(s_1, \ldots, s_n)
\ = \ 				\integral{\infty}
				\phi(t_1, \ldots, t_n)
				e(-(s_1t_1 + \ldots + s_nt_n))\,\ud t_1 \cdots \,\ud t_n.
\end{align}
We have
\begin{align}
				\hat{\psi}(\nu_1, \ldots, \nu_n)
\ = \ 				\frac{1}{(d - 1)^n}
				\hat{\phi}\pa{\frac{\nu_1}{d - 1}, \ldots, \frac{\nu_n}{d - 1}}
\end{align}
and
\begin{align}
				F_{n, \chi}(\psi)
&\ = \ 			\frac{1}{(d - 1)^n}
				W_{n, \chi}(\phi)
\end{align}
for all $n$.
Thus we obtain the following local-regime results.

\begin{cor}\label{6.2}
Suppose $Q$ is irreducible.
Let $\phi$ be a $1$-dimensional test function of rapid decay.
Let
\begin{align}
					C(\phi; d)
&\ \dfeq\ 			\sum_{\nu \in \Z}
					\left|\hat{\phi}
					\pa{\frac{\nu}{d - 1}}\right|
					q^{|\nu|/2}.
\end{align}
\begin{enumerate}
\item 	If $C(\phi; d) \ll dq^d$ as $d \to \infty$, then
\begin{align}
					\ebb W_{1,\chi}(\phi)
&\ = \ 				\hat{\phi}(0)
					-		\frac{1}{(d - 1)(q - 1)}
							\sum_{\nu \in \Z}
							\frac{1}{q^{|\nu|/2}}
							\hat{\phi}\pa{\frac{\nu}{d - 1}}
					+		O\pa{
							\frac{C(\phi)}{dq^d}}.
\end{align}
\item 	If $C(\phi; d) \ll dq^{d/2}$ as $d \to \infty$, then
\begin{align}
					\Var W_{1,\chi}(\phi)
&\ = \ 				\frac{1}{(d - 1)^2}
					\sum_{\nu \in \Z}
					|\nu| \left|\hat{\phi}\pa{\frac{\nu}{d - 1}}\right|^2
					+		O\pa{
							\frac{C(\phi)^2}{d^2q^d}
							}.
\end{align}
\end{enumerate}
\end{cor}

\begin{cor}\label{6.3}
Suppose $Q$ is irreducible.
Let $\phi$ be a $2$-dimensional test function of rapid decay such that $C(\phi_1; d), C(\phi_2; d) \ll dq^{d/2}$, where $C(\phi; d)$ is defined for $1$-dimensional test functions $\phi$ of rapid decay in Corollary \ref{6.2}.
Let $C(\phi) = C(\phi_1)C(\phi_2)$ and $\phi_\diag(s) = \phi(s, s)$.
Then
\begin{align}
					\ebb W_{2, \chi}(\phi)
&\ = \ 				{-\ebb W_{1,\chi}(\phi_\diag) }
					+		\hat{\phi}(0, 0)
					+		\frac{1}{(d - 1)^2}
							\sum_{\nu \in \Z}
							|\nu|
							\hat{\phi}\pa{\frac{n}{d - 1}, -\frac{n}{d - 1}}
							\\
&\qquad		+		\frac{C_{2, \Gamma}(\phi; d)}{q - 1}
					+ 		O\pa{\frac{C(\phi_1; d) + C(\phi_2; d)}{d q^d}}
					+		O\pa{\frac{C(\phi; d)}{d^2 q^d}},
					\nonumber
\end{align}
where
\begin{align}
C_{2, \Gamma}(\phi; d)
&\ = \ 				{-\frac{1}{d - 1}}
					\sum_{\nu \in \Z}
					\frac{1}{q^{|\nu|/2}}
					\pa{\hat{\phi}{\pa{0, \frac{\nu}{d - 1}}} + \hat{\phi}{\pa{\frac{\nu}{d - 1}, 0}}}
					\\
&\qquad
					+
					\frac{1}{(d - 1)^2}
					\sum_{\nu_1, \nu_2 \in \Z}
					\frac{1}{q^{(|\nu_1| + |\nu_2|)/2}}
					\hat{\phi}{\pa{\frac{\nu_1}{d - 1}, \frac{\nu_2}{d - 1}}}.
					\nonumber
\end{align}
\end{cor}

\begin{rem}
In Corollary \ref{6.2}, if $\supp \hat{\phi} \subeq [-2, 2]$, then the hypothesis of (1) is satisfied, and if $\supp \hat{\phi} \subeq [-1, 1]$, then the hypothesis of (2) is satisfied.
In both cases, our results match the unitary predictions of \cite{HR} in the limit.
More precisely our results: (1) implies the $\F_q(T)$-analogue of their Theorem 3.1, and (2) implies the $\F_q(T)$-analogue of their Theorem 3.4.
Similarly, in Corollary \ref{6.3}, if $\supp \hat{\phi}_1, \supp \hat{\phi}_2 \subeq [-1, 1]$, then our result matches the prediction of \cite{KS1}.
\end{rem}

%%%%%%%%%%%%%%%%%%%%%%%%%%%%%%%%%%%%%%%%%%%%%%%%%%%%%%%%%%%%%%%%%%%%%%%%%%%%%%%%%%%%%%%%%%%%%%%%%%%%%%%%%%%%%%%%%%%%%%%%%%%%%
%%%%%%%%%%%%%%%%%%%%%%%%%%%%%%%%%%%%%%%%%%%%%%%%%%%%%%%%%%%%%%%%%%%%%%%%%%%%%%%%%%%%%%%%%%%%%%%%%%%%%%%%%%%%%%%%%%%%%%%%%%%%%
%%%%%%%%%%%%%%%%%%%%%%%%%%%%%%%%%%%%%%%%%%%%%%%%%%%%%%%%%%%%%%%%%%%%%%%%%%%%%%%%%%%%%%%%%%%%%%%%%%%%%%%%%%%%%%%%%%%%%%%%%%%%%
\section{Montgomery's Hypothesis}\label{sec:montgomeryhyp}

Returning to the classical setting, Fiorilli and Miller \cite{FioM} showed how to relate certain conjectures about the distribution of  primes in arithmetic progressions to improvements in the available support for $\hat{\phi}$ in the local-regime density results of Hughes-Rudnick. One that generalizes to our setting is a weakened version of Montgomery's Hypothesis, originally stated in \cite{Mo1}.

Let $Q \in \Z_+$. Let $\Lambda$ be the classicial von Mangoldt function, and let $\Psi(X) = \sum_{n \le X} \Lambda(n)$, the Chebyshev function.
For all $a \in \Z_+$ coprime to $Q$, let
\begin{align}
					\Psi(X;Q,a)
\ = \  					\sum_{\substack{n \le X \\ n \equiv a \pmod{Q}}}
					\Lambda(n).
\end{align}

\begin{conj}[$\theta$-Montgomery]\label{7.1}
Let $\Phi$ be the classical Euler totient function.
Then for all $Q \geq 3$, there exists $\theta \in (0, 1/2]$ such that
\begin{align}
\Psi(X; Q, 1)
-		\frac{\Psi(X)}{\Phi(Q)}
\ \ll_\epsilon\
					\frac{X^{1/2 + \epsilon}}{Q^\theta}
\end{align}
for all $X \geq Q$.
\end{conj}

Although we do not expect the conjecture to hold for $\theta = 1/2$, it is likely to hold for any arbitrarily smaller value.
Theorem 1.16 of \cite{FioM} implies the following.

\begin{thm}[Fiorilli-Miller]
Let $Q \geq 3$ be prime, and let $\fcal_Q$ be the family of primitive Dirichlet $L$-functions of conductor $Q$.
Let $\phi$ be a $1$-dimensional test function of rapid decay such that $\hat{\phi}$ is compactly support.
If Conjecture \ref{7.1} holds for $\theta$, then
\begin{align}
					\frac{1}{\#\fcal_Q}
					\sum_{\chi \in \fcal_Q}
					\sum_{\gamma_\chi}
					\phi
					\pa{\gamma_\chi \frac{\log Q}{2\pi}}
&\ = \ 				\hat{\phi}(0)
					\ +\
					{\rm Gamma-factor\ term}
					\ +\
					O(Q^{-\theta + \epsilon}).
\end{align}
\end{thm}

That is, $\theta$-Montgomery implies that the $1$-level density of $\fcal_Q$ tends to that of the unitary group for \emph{all} $\phi$ such that $\hat{\phi}$ has compact support, and the error term improves exponentially with $\theta$.

We return to the function-field setting.
Let $\Lambda_q$ be the von Mangoldt function for $\F_q(T)$, and let ${\Psi_q(n) = \sideset{}{^{'}}\sum_{\deg f = n}} \Lambda_q(f)$.
(Note how this differs from the most na\"ive analogue of the Chebyshev function.)
For all nonconstant $Q \in \F_q[T]$ and $f\in \F_q[T]$ coprime to $Q$, let
\begin{align}
					\Psi_q(n; Q, f)
\ = \ 					{\sideset{}{^{'}}\sum_{\substack{\deg g = n \\ g \equiv f\pmod{Q}}}}
					\Lambda_q(g).
\end{align}
Let $\fcal_Q$ resume its definition from Section 4.
Recall from the proof of Theorem \ref{thm:1levelexp} that if $\psi$ is a $1$-dimensional test function of period $1$ and $Q$ is irreducible of degree $\geq 2$, then $\ebb F_{1, \chi}(\psi) = \hat{\psi}(0) + \text{Gamma-factor term} + \ebb (F_{1, \chi}(\psi)^\osc)$, where
\begin{align}
					\ebb (F_{1, \chi}(\psi)^\osc)
\ = \ 					-\frac{1}{d - 1}
					\sum_{n = 0}^\infty
					\pa{\Psi_q(n; Q, 1)
					-
					\frac{\Psi_q(n)}{\# \fcal_Q}}
					\frac{\hat{\psi}(n) + \hat{\psi}(-n)}{q^{n/2}}.
\end{align}
Thinking of $q^n$ as the correct analogue of the $X$ variable in Conjecture \ref{7.1}, we are led to the following conjecture.

\begin{conj}[$\theta$-Montgomery for $\F_q(T)$]\label{7.3}
Let $\Phi_q(f) = \#(\F_q[T]/f)^\times$ be the Euler totient function for $\F_q[T]$.
Then for all $Q \in \F_q[T]$ of degree $d \geq 2$, there exists $\theta \in (0, 1/2]$ such that
\begin{align}
					\Psi_q(n; Q, 1)
					-
					\frac{\Psi_q(n)}{\Phi_q(Q)}
\ \ll_\epsilon\
					q^{n(1/2 + \epsilon) - d\theta}
\end{align}
for all $n \geq d$.
\end{conj}

We remark that if Montgomery's Hypothesis is translated from the language of primes to the language of zeros, guided by the duality that exists between primes and zeros of $L$-functions, then we obtain a conjecture that relates to $F_{1, \chi}$ more naturally.
For all $\chi \in \fcal_Q$, let $\{\gamma_\chi, j\}_{j = 1}^{d(\chi)}$ be the ordinates of the zeros of $L(s, \chi)$.
We propose the following.

\begin{conj}\label{7.4}
Let $Q \in \F_q[T]$ be of degree $d \geq 2$.
Then there exist $\theta_1, \theta_2 \in (0, 1)$ such that for all $n \in \Z_+$,
\begin{align}
					\sum_{\chi \in \fcal_Q}
					\sum_j
					q^{in\gamma_{\chi, j}}
\ \ll_{\theta_1, \theta_2}\
					d(\chi)^{(1 - \theta_1)}
					q^{d(1 - \theta_2)}
\end{align}
as $d \to \infty$ (where $d(\chi) = d - 1 - \lambda_\infty(\chi)$ is the degree of $L(s, \chi)$ seen as a polynomial in the variable $q^{-s}$).
\end{conj}

\begin{prop}\label{7.5}
Let $Q \in \F_q[T]$ be irreducible of degree $d \geq 2$.
Let $\psi$ be a $1$-dimensional test function of period $1$.
\begin{enumerate}
\item 	Suppose Conjecture \ref{7.3} holds for some $\theta$.
			If $C_\epsilon(\psi) = \sum_{n \in \Z} \hat{\psi}(n) q^{|n|\epsilon}$ converges for all $\epsilon > 0$ small enough, then
\begin{align}
				\ebb (F_{1, \chi}(\psi)^\osc)
\ \ll_\epsilon\
				\frac{C_\epsilon(\psi)}{dq^{d\theta}}.
\end{align}
\item 	Suppose Conjecture \ref{7.4} holds for some $\theta_1, \theta_2$.
			Then
			\begin{align}
				\ebb (F_{1, \chi}(\psi)^\osc)
\ \ll_{\theta_1, \theta_2}\
				\frac{\theta_2}{d^{\theta_1 - 1} (\#\fcal_Q)^{\theta_2}}.
			\end{align}
\end{enumerate}
\end{prop}

\begin{proof}
(1) is immediate.
(2) follows from the Erd\H os-Tur\' an Inequality, which, together with Conjecture \ref{7.4}, implies that for all $[a, b] \subeq [-T_q/2, +T_q/2]$ and $N \in \Z_+$,
\begin{align}
				\left|
				\frac{\#\{(\chi, j) : \gamma_{\chi, j} \in [a, b]\}}{(d - 1)\# \fcal_Q}
				-	
				\frac{b - a}{T_q}
				\right|
&\ \ll \ 			\frac{1}{N}
				+
				\sum_{n = 1}^N
				\frac{1}{n}
				\left|
				\frac{1}{(d - 1)\#\fcal_Q}
				\sum_\chi
				\sum_j
				q^{in\gamma_{\chi, j}}
				\right|
				\\
&\ = \ 			\frac{1}{N}
				+
				\frac{1}{d^{\theta_1}(\#\fcal_Q)^{\theta_2}}
				\sum_{n = 1}^N
				\frac{1}{n}.
				\nonumber
\end{align}
The supremum of the expression on the left over all $a, b$ is an upper bound for $\ebb (F_{1, \chi}(\psi)^\osc)$, as we can approximate $\psi$ arbitrarily well by linear combinations of indicator functions.
Choosing $N = \lfloor d^{\theta_1} (\#\fcal_Q)^{\theta_2}\rfloor$ completes the proof.
\end{proof}

Either of the two possibilities suggested by Proposition \ref{7.5} is considerably stronger than Theorem \ref{thm:1levelexp}.
They imply the following results in the local regime.

\begin{cor}
Let $Q \in \F_q[T]$ be irreducible of degree $d \geq 2$.
Let $\phi$ be a $1$-dimensional test function of rapid decay such that $\hat{\phi}$ has compact support.
\begin{enumerate}
\item 	Suppose Conjecture \ref{7.3} holds for some $\theta$.
			Then the hypotheses in Corollary \ref{6.2} can be lifted and the error term can be sharpened to $O_\epsilon(q^{-d(\theta - \epsilon)})$.
\item 	Suppose Conjecture \ref{7.4} holds for some $\theta_1, \theta_2$.
			Then the hypotheses in Corollary \ref{6.2} can be lifted and the error term can be sharpened to $O_{\theta_1, \theta_2}(d^\theta_2 d^{1 - \theta_1}(\#\fcal_Q)^{-\theta_2})$.
\end{enumerate}
\end{cor}

\appendix
%%%%%%%%%%%%%%%%%%%%%%%%%%%%%%%%%%%%%%%%%%%%%%%%%%%%%%%%%%%%%%%%%%%%%%%%%%%%%%%%%%%%%%%%%%%%%%%%%%%%%%%%%%%%%%%%%%%%%%%%%%%%%%%%%%%%%%%%%%%%
%%%%%%%%%%%%%%%%%%%%%%%%%%%%%%%%%%%%%%%%%%%%%%%%%%%%%%%%%%%%%%%%%%%%%%%%%%%%%%%%%%%%%%%%%%%%%%%%%%%%%%%%%%%%%%%%%%%%%%%%%%%%%%%%%%%%%%%%%%%%
%%%%%%%%%%%%%%%%%%%%%%%%%%%%%%%%%%%%%%%%%%%%%%%%%%%%%%%%%%%%%%%%%%%%%%%%%%%%%%%%%%%%%%%%%%%%%%%%%%%%%%%%%%%%%%%%%%%%%%%%%%%%%%%%%%%%%%%%%%%%
\section{Proof of the Explicit Formula}\label{sec:proofprop3point1}

\begin{proof}[Proof of Proposition \ref{prop:3.1}]
Let
\begin{align}
\psi^\ast(s) \ = \  \frac{1}{d - 1}\psi\pa{\frac{s}{T_q}}.
\end{align}
For all real $c$, let $\ell_c$ be the segment from $c - iT_q/2$ to $c + iT_q/2$ in the complex plane.
Let $0 < \epsilon < 1/4$ and $c = 1/2 + \epsilon$.
Using Cauchy's Theorem,
\begin{align}
					\sum_{-\frac{T_q}{2} \le \gamma_\chi < \frac{T_q}{2}}
					\psi^\ast(\gamma_\chi)
&\ = \ 				\frac{1}{2\pi i}
					\pa{\int_{\ell_c} - \int_{\ell_{1 - c}}}
					\frac{L'}{L}(s, \chi)\psi^\ast(-i(s - 1/2))\,\ud s +  O(\epsilon)
					\\
&\ = \ 				F_{1, \chi}(\psi; \epsilon) - F_{1, \chi}(\psi; -\epsilon) + O(\epsilon),
					\nonumber
\end{align}
where
\begin{align}
					F_{1, \chi}(\psi; \epsilon)
\ = \ 					\frac{1}{2\pi}
					\integral{T_q/2}
					\frac{L'}{L}(1/2 + \epsilon + it, \chi)\psi^\ast(t - i\epsilon)\,\ud t.
\end{align}
To deal with $F_{1, \chi}(\psi; -\epsilon)$ we substitute the formula (\ref{eq_2.9}).
Distributing the integral among the resulting three terms, and sending $\epsilon \to 0$ in the first and last, we find
\begin{align}\label{eq_3.6}
					-F_{1, \chi}(\psi; -\epsilon)
					\\
&\ = \ 				\frac{d - 1}{T_q}
					\integral{T_q/2}
					\psi^\ast(t)\,\ud t
					+		\frac{1}{2\pi}
							\integral{T_q/2}
							\frac{L'}{L}(1/2 + \epsilon - it, \bar{\chi})\psi^\ast(t + i\epsilon)\,\ud t
					\nn
&\qquad		+		\frac{\lambda_\infty(\chi)}{T_q}
							\integral{T_q/2}
							\pa{-1 + \frac{1}{1 - q^{1/2 - it}} + \frac{1}{1 - q^{1/2 + it}}}
							\psi^\ast(t)
							\,\ud t.
							\nonumber
\end{align}
The first term of the right side is $\hat{\psi}(0) = \integral{1/2} \psi(t)\,\ud t$, while the last term equals
\begin{align}\label{eq_3.7}
&					\frac{\lambda_\infty(\chi)}{(d - 1)T_q}
					\integral{T_q/2}
					\pa{-1 + \frac{1}{1 - q^{1/2 - it}} + \frac{1}{1 - q^{1/2 + it}}}
					\psi\pa{\frac{t}{T_q}}
					\,\ud t
					\\
&\ = \ 				\frac{\lambda_\infty(\chi)}{(d - 1)T_q}
					\pa{-1 - \frac{q^{-1/2 + it}}{1 - q^{1/2 - it}} - \frac{q^{-1/2 - it}}{1 - q^{1/2 + it}}}
					\psi\pa{\frac{t}{T_q}}
					\,\ud t
					\nn
&\ = \ 				-\frac{\lambda_\infty(\chi)}{d - 1}
					\sum_{n \in \Z}
					\frac{\hat{\psi}(n)}{q^{|n|/2}}.
					\nonumber
\end{align}
Finally,
\begin{align}
					F_{1, \chi}(\psi; \epsilon)
\ = \ 					\frac{1}{(d - 1)T_q}
					\sum_{n = 0}^\infty
					\integral{T_q/2}
					\frac{c_\chi(n)}{q^{n(1/2 + \epsilon + it)}} \psi\pa{\frac{t - i\epsilon}{T_q}}\,\ud t,
\end{align}
and similarly with the middle term of (\ref{eq_3.6}), where our use of the Dirichlet series in the region $\Re s > 1/2$ is justified by the bound $c_\chi(n) \ll dq^{n/2}$ from the Riemann Hypothesis.
Interchanging the sum with the integral, and sending $\epsilon \to 0$, we arrive at the desired result.
\end{proof}

Interestingly, the middle term on the right side of (\ref{eq_3.2}) corresponds to the Gamma-factor term in the classical Explicit Formula, but is visually much simpler.
This is because the $\F_q(T)$-analogue of the Riemann zeta function and of the Gamma function are itself simpler objects.

\section{Traces of the Frobenius Class}

We use this appendix to interpret the main results of
this paper in terms of averages of traces of powers
of random unitary matrices.

Let us assume that $\chi$ is a odd and primitive
Dirichlet character modulo $Q$ over $\mathbb{F}_{q}(T)$,
with $\mathrm{deg}(Q)=d$. Then the associated
$L$--function can be written as

\begin{equation}
L(s,\chi)=\det(I-u\sqrt{q}\Theta_{\chi}),
\end{equation}
where $\Theta_{\chi}$ is a unitary matrix (or rather, the conjugacy class of unitary matrices) and is called the unitarized Frobenius matrix of $\chi$. We also denote it by $\Theta_{\chi}=\mathrm{diag}(e^{i\theta_{1}},\ldots,e^{i\theta_{N}})$, with the $e^{i\theta_{j}}$'s being the eigenvalues of $\Theta_{\chi}$.

The work of Katz and Sarnak \cite{KS1} shows that as $q\rightarrow\infty$, the Frobenius classes $\Theta_{\chi}$ become equidistributed in the unitary group $U(d-1)$. This implies that various statistics of the eigenvalues can, in this limit, be computed by integrating the corresponding quantities over $U(d-1)$. The goal of this paper was to explore the opposite limit, that of fixed constant field and large $d$. In this paper we have studied the basic cases of the expected values of powers of the traces of $\Theta_{\chi}$.

The mean value of traces of powers when averaged over the unitary group $U(d-1)$ was computed by Diaconis and Shahshahani in \cite{DS} and by Dyson in \cite{Dys},

\begin{equation}
 \int_{U(d-1)}\mathrm{Tr}(A^{n})dA =
  \begin{cases}
   d-1 & \text{if } n=0, \\
   0      & \text{otherwise.}
  \end{cases}
\end{equation}

With this notation, Theorem \ref{thm:1levelexp} shows that
for fixed $n$

\begin{equation}
\mathbb{E}(\mathrm{Tr}(\Theta_{\chi}^{n}))\sim\int_{U(d-1)}\mathrm{Tr}(A^{n})dA.
\end{equation}

Likewise Theorem \ref{thm:1levelvar} shows that for fixed $n$

\begin{equation}
\mathbb{E}(|\mathrm{Tr}(\Theta_{\chi}^{n})|^{2})\sim\int_{U(d-1)}|\mathrm{Tr}(A^{n})|^{2}dA,
\end{equation}
where Diaconis and Shahshahani \cite{DS} showed that

\begin{equation}
\int_{U(d-1)}|\mathrm{Tr}(A^{n})|^{2}dA=\mathrm{min}(|n|,d-1).
\end{equation}

Last, Theorem \ref{thm:2levelexp} shows that for fixed $n$

\begin{equation}
\mathbb{E}(\mathrm{Tr}(\Theta_{\chi}^{j})\mathrm{Tr}(\Theta_{\chi}^{-k}))\sim\int_{U(d-1)}\mathrm{Tr}(A^{j})\mathrm{Tr}(A^{-k})dA,
\end{equation}
and Diaconis and Shahshahani \cite{DS} showed that

\begin{equation}
\int_{U(d-1)}\mathrm{Tr}(A^{j})\mathrm{Tr}(A^{-k})dA=\delta_{j,k}
\begin{cases}
   (d-1)^{2} & \text{if } k=0, \\
   |k|      & \text{if } 1\leq|k|\leq d-1, \\
	 d-1      & \text{if } |k|>d-1.
  \end{cases}
\end{equation}

%%%%%%%%%%%%%%%%%%%%%%%%%%%%%%%%%%%%%%%%%%%%%%%%%%%%%%%%%%%%%%%%%%%%%%%%%%%%%%%%%%%%%%%%%%%%%%%%%%%%%%%%%%%%%%%%%%%%%%%%%%%%%%%%%%%%%%%%%%%%
%%%%%%%%%%%%%%%%%%%%%%%%%%%%%%%%%%%%%%%%%%%%%%%%%%%%%%%%%%%%%%%%%%%%%%%%%%%%%%%%%%%%%%%%%%%%%%%%%%%%%%%%%%%%%%%%%%%%%%%%%%%%%%%%%%%%%%%%%%%%
%%%%%%%%%%%%%%%%%%%%%%%%%%%%%%%%%%%%%%%%%%%%%%%%%%%%%%%%%%%%%%%%%%%%%%%%%%%%%%%%%%%%%%%%%%%%%%%%%%%%%%%%%%%%%%%%%%%%%%%%%%%%%%%%%%%%%%%%%%%%

\ \\

\end{document}